\documentclass[draft]{amsart}

\usepackage{amssymb}
\usepackage{url}
\usepackage{amscd}

\theoremstyle{plain}
\newtheorem{theorem}{Theorem}[section]

\newtheorem{corollary}[theorem]{Corollary}

\newtheorem{proposition}[theorem]{Proposition}
\newtheorem{question}[theorem]{Question}

\theoremstyle{remark}

\begin{document}

\date{}

\title{Lebesgue measurability of separately continuous functions
and separability}

\author{V.V. Mykhaylyuk}

\begin{abstract}
It is studied a connection between the separability and the
countable chain condition of spaces with the $L$-property (a
topological space $X$ has the $L$-property if for every
topological space $Y$, separately continuous function $f:X\times
Y\to\mathbb R$ and open set $I\subseteq \mathbb R$ the set
$f^{-1}(I)$ is a $F_{\sigma}$-set). We show that every completely
regular Baire space with the $L$-property and the countable chain
condition is separable and construct a nonseparable completely
regular space with the $L$-property and the countable chain
condition. This gives a negative answer to a question of M.~Burke.
\end{abstract}

\maketitle


Chernivtsi National University, Department of Mathematical
Analysis,

Kotsjubyns'koho 2, Chernivtsi 58012, Ukraine,  mathan@chnu.cv.ua

2000 Mathematics Subject Classification 54C05, 54C10, 54E52

\newcommand\sfrac[2]{{#1/#2}}

\section{Introduction}

A function $f:X\to\mathbb R$ defined on a topological space $X$ is
called {\it a first Baire class function} if there exists a
sequence $(f_n)^{\infty}_{n=1}$ of continuous functions
$f_n:X\to\mathbb R$ which converges pointwise to $f$ on $X$; and
{\it a first Lebesgue class function} if $f^{-1}(G)$ is a
$F_{\sigma}$-set for every open set $G\subseteq \mathbb R$.
Standard reasons (see \cite[p.~394]{Ku}) shows that every first
Baire class function is a first Lebesgue class function.

Investigations of Baire and Lebesgue classifications of separately
continuous functions were started by H.~Lebesgue in \cite{L} and
were continued in papers of many mathematicians (see \cite{MMMS}).

We say that a topological space $X$ has {\it the $B$-property}
({\it the $L$-property}) if for every topological space $Y$ each
separately continuous function $f:X\times Y\to\mathbb R$ is a
first Baire class function (a first Lebesgue class function).

It is known \cite{Ba,Bu1} that any topological space $X$ has the
$B$-property (the $L$-property) if and only if the evaluation
function $c_X:X\times C_p(X)\to\mathbb R$, $c_X(x,y)=y(x)$, is a
first Baire class function (a first Lebesgue class function),
where $C_p(X)$ means the space of continuous on $X$ functions with
the pointwise convergence topology.

Baire and Lebesgue classifications of separately continuous
function were investigated in \cite{Bu2}. In particular, it was
shown in \cite{Bu2} that any completely regular space $X$ with the
$B$-property and the countable chain condition is separable
(topological space $X$ has a countable chain condition (C.C.C.) if
every system of disjoint open in $X$ sets is at most countable).
In this connection the following question arose  in \cite[Problem
4.6]{Bu2}.

\begin{question}
{\it Is every completely regular space $X$ with the $L$-property
and the countable chain condition a separable space?}
\end{question}

In this paper we show that if a space $X$ is a Baire space then
Question 1.1 has a positive answer and construct an example which
gives a negative answer to the question in general case.

\section{Density of Baire spaces with the $L$-property}

The minimal cardinal $\aleph \geq \aleph_0$ for which any system
of disjoint open in a topological space $X$ sets has the
cardinality at most $\aleph$ is called {\it a Souslin number of
$X$} and is denoted by $c(X)$. Note that the countable chain
condition of $X$ means that $c(X)=\aleph_0$. It is easy to see
that $c(X)\leq d(X)$, where $d(X)$ is the density of $X$.

The following result implies that for a Baire space $X$ Question
1.1 has a positive answer.

\begin{theorem}
{\it Let $X$ be a completely regular Baire space with the
$L$-property. Then $c(X)=d(X)$.}
\end{theorem}

\begin{proof}
Since the evaluation function $c_X$ is a first Lebesgue class
function, the set $E=\{(x,y):y(x)=0\}$ is a $G_{\delta}$-set in
$X\times Y$, where $Y=C_p(X)$. Choose a sequence
$(W_n)^{\infty}_{n=1}$ of open in $X\times Y$ sets $W_n$ such that
$E=\bigcap\limits_{n=1}^{\infty} W_n$. Denote by $y_0$ the
null-function on $Y$. For every $n\in \mathbb N$ and an $x\in X$
find open neighborhoods $U(x,n)$ and $V(x,n)$ of $x$ and $y_0$ in
$X$ and $Y$ respectively such that $U(x,n)\times V(x,n)\subseteq
W_n$.

Fix a $n\in \mathbb N$ and show that there exists a set
$A_n\subseteq X$ with $|A_n|\leq c(X)=\aleph$ such that the open
set $G_n=\bigcup\limits_{x\in A_n} U(x,n)$ is dense in $X$.
Consider a system $\mathcal U$ of all open in $X$ nonempty sets
$U$ such that $U\subseteq U(x,n)$ for some $x\in X$ and choose a
maximal system ${\mathcal U}'\subseteq \mathcal U$ which consists
of disjoint sets. It is clear that $|{\mathcal U}'|\leq \aleph$.
For every $U\in {\mathcal U}'$ find an $x=x(U)\in X$ such that
$U\subseteq U(x,n)$ and put $A_n=\{x(U):U\in {\mathcal U}'\}$.
Then $|A_n|\leq |{\mathcal U}'|\leq \aleph$. Besides, it follows
from the maximality of ${\mathcal U}'$ that $G_n$ is dense in $X$.

Since $X$ is a Baire space, the set
$X_0=\bigcap\limits_{n=1}^{\infty} G_n$ is dense in $X$. For every
$n\in\mathbb N$ and $x\in X$ choose a finite set $B(x,n)\subseteq
X$ such that $y\in V(x,n)$ for each $y\in Y$ with
$y|_{B(x,n)}=y_0|_{B(x,n)}$. Put $B=\bigcup\limits_{n\in\mathbb N}
\bigcup\limits_{x\in A_n} B(x,n)$. Note that $|B|\leq
\aleph_0\cdot \aleph=\aleph$.

Show that $B$ is dense in $X$. Since $X$ is a completely regular
space it is enough to prove that $y_0$ is a unique continuous on
$X$ function which equals to 0 at every point from $B$. Let $y\in
Y$ be a function such that $y(b)=0$ for every $b\in B$. Fix a
point $x\in X_0$ and an integer $n\in\mathbb N$. Find $a\in A_n$
such that $x\in U(a,n)$. Then $B(a,n)\subseteq B$ implies $y\in
V(a,n)$. Therefore $(x,y)\in W_n$. Thus $X_0\times\{y\}\subseteq
\bigcap\limits_{n=1}^{\infty} W_n =E$, that is $y(x)=0$ for every
$x\in X_0$. Hence $y=y_0$ because $X_0$ is dense in $X$.

Thus $d(X)\leq |B|\leq c(X)$. Therefore $c(X)=d(X)$.
\end{proof}

\begin{corollary}
{\it Every completely regular Baire space with the $L$-property
and the countable chain condition is a separable space.}
\end{corollary}

\section{Nonseparable spaces with the $L$-property and C.C.C.}

The following notion was introduced in \cite{Ba}, where some
properties of spaces with the $B$-property were studied.

A topological space $X$ with a topology $\tau$ is called {\it
quarter-stratifiable} if there exists a function $g:\mathbb
N\times X\to \tau$ such that

$(i)\,\,\,\,X=\bigcup\limits_{x\in X} g(n,x)$ for every $n\in
\mathbb N$;

$(ii)\,\,\,$ if $x\in g(n,x_n)$ for each $n\in \mathbb N$ then
$x_n\to x$.

The following result follows from \cite[Proposition 2.1]{Ka}.

\begin{proposition}
{\it Every quarter-stratifiable space $X$ has the $L$-property.}
\end{proposition}

A topological space $X$ is called {\it $\sigma$-discrete} if there
exists an increasing sequence $(X_n)^{\infty}_{n=1}$ of closed
discrete subspaces $X_n$ of $X$ such that
$X=\bigcup\limits_{n=1}^{\infty} X_n$.

\begin{proposition}
{\it Every $\sigma$-discrete space is a quarter-stratifiable
space.}
\end{proposition}

\begin{proof}
Let $(X_n)_{n=1}^{\infty}$ be an increasing sequence of closed
discrete subspaces $X_n$ of $X$ such that
$X=\bigcup\limits_{n=1}^{\infty} X_n$. For every $n\in\mathbb N$
and $x\in X_n$ denote by $U(x,n)$ an open in $X$ neighborhood of
$x$ such that $U(x,n)\cap X_n=\{x\}$. A function $g:\mathbb
N\times X\to \tau$, where $\tau$ is the topology of $X$, define
letting: $g(x,n)=U(x,n)$ if $x\in X_n$ and $g(x,n)=X\setminus X_n$
if $x\not \in X_n$. It is easy to see that $g$ satisfies $(i)$ and
$(ii)$.
\end{proof}

Show now that Question 1.1 has a negative answer.

\begin{theorem}
{\it There exists a completely regular nonseparable space with the
$L$-proper\-ty and with the countable chain condition.}
\end{theorem}

\begin{proof}
Let $\Gamma_0$ be a set with $|\Gamma_0|\geq \aleph_1$,
$(a_n)^{\infty}_{n=1}$ be a sequence of distinct points
$a_n\not\in \Gamma_0$, $\Gamma_n=\Gamma_0\cup\{a_k:1\leq k \leq
n\}$, ${\mathcal A}_n$ be a system of all subsets $A\subseteq
\Gamma_{n-1}$ such that $|A|=n$. Denote by $X_n$ a set of all
function $x\in \{0,1\}^{\Gamma}$ such that $x=\chi_{A\cup\{a_n\}}$
for some $A\in {\mathcal A}_n$, where $\chi_B$ means the
characteristic function of $B$, and put
$X=\bigcup\limits_{n=1}^{\infty} X_n$.

Show that $X$ is a $\sigma$-discrete space. For every $n\in\mathbb
N$ put $Y_n=\bigcup\limits_{k=1}^{n} X_k$. Fix an integer
$n\in\mathbb N$ and for each $1\leq k\leq n$ put $G_k=\{x\in X:
x(a_k)=1, x(a_i)=0, k< i\leq n\}$. It is easy to see that $G_k\cap
Y_n=X_k$. Since all spaces $X_k$ are discrete, $Y_n$ is discrete
in $X$ too. Besides, $Y_n$ is closed in $X$. Thus, $X$ has the
$L$-property by Propositions 3.1 and 3.2.

Note that $X$ is dense in $Y=\{0,1\}^{\Gamma}$. Indeed, let
$A\subseteq \Gamma$ be a finite set and $y:A\to\{0,1\}$. Choosing
$n\geq |A|$  with $A\subseteq \Gamma_n$ find $x\in X_{n+1}$ such
that $x|_A=y$. Then $c(X)=\aleph_0$ since $c(Y)=\aleph_0$ and $X$
is dense in $Y$.

It remains to note that $X$ is nonseparable because for every
separable subspace $Z$ of $X$ there exists a countable set
$B\subseteq \Gamma$ such that $z(\gamma)=0$ for every $\gamma\in
\Gamma\setminus B$.
\end{proof}

This example shows that there exists a quarter-stratifiable space
which has not the $B$-property. Thus, Proposition 3.1 cannot be
generalized for spaces with the $B$-property.

A family $(A_i:i\in I)$ of sets $A_i$ is called {\it pointwise
finite} if $\bigcap\limits_{i\in J} A_i=\O$ for each infinite set
$J\subseteq I$. A cardinal
$$
p(X)= \sup \{|\mathcal A|: {\mathcal
A}\mbox{\,\,is\,a\,\,pointwise\,\,finite\,\,family\,\,of\,\,nonempty
\,open\,\,in\,\,}X \mbox{\,\,sets}\}
$$
is called {\it a point-finite cellularity of a topological space
$X$}. Clearly $c(X)\leq p(X)$. Besides, it is known that
$p(X)=c(X)$ for each Baire space $X$. Therefore the following
question arises naturally from Theorem 2.1 and the fact that
$p(X)=|\Gamma|>\aleph_0$ for the space $X$ from Theorem~3.3.

\begin{question}
{\it Is every completely regular space $X$ with the $L$-property
and $p(X)=\aleph_0$ a separable space?}
\end{question}

\end{document}